\documentclass[12pt, a4paper]{amsart}

\date{19th March 2021}
\date{\today} 

\usepackage{amsmath, amssymb, comment, amsthm, amscd}
\usepackage{graphicx}
\usepackage{tikz-cd}
\usepackage[inline]{enumitem}
\usepackage{curves}
\usepackage{pict2e}
\usetikzlibrary{intersections,calc,arrows}
\usepackage{ytableau}
\usepackage{autobreak}

\newtheorem{thm}{Theorem}[section] 
\newtheorem{dfn}[thm]{Definition} 
\newtheorem{pro}[thm]{Proposition} 
\newtheorem{cor}[thm]{Corollary} 
\newtheorem{con}[thm]{Conjecture} 
\newtheorem{lem}[thm]{Lemma}

\def\tc{\operatorname{tc}}

\def\TC{\operatorname{TC}}

\def\cat{\operatorname{cat}}

\def\integral{\mathbb{Z}}
\def\homeo{\approx}
\def\StrColl{\searrow\hskip-.6em\searrow}
\def\eStrColl{\searrow\hskip-.6em\searrow{\vphantom{(}}^{\hskip-.7em e}\hskip.35em}
\def\StrExp{\nearrow\hskip-.6em\nearrow}
\def\eStrExp{\nearrow\hskip-.6em\nearrow{\vphantom{(}}_{\hskip-.7em e}\hskip.35em}
\def\double#1{#1 \!\times\! #1}
\newcommand{\Cir}{\mathbb{S}}

\newcommand{\Height}{\operatorname{ht}}

\def\Angle#1{[#1]}

\begin{document}
\baselineskip21pt

\title{Topological complexity of Khalimsky circles}
\author{Ryusei Yoshise}
\email{yoshise.ryusei.597@s.kyushu-u.ac.jp}
\address{Joint Graduate School of Mathematics for Innovation, Kyushu University, Fukuoka 819-0395, Japan}  

\begin{abstract}
We determine topological complexity of a series of finite spaces which is weakly homotopy equivalent to a circle $S^1$, and give a finite space $X$ satisfying the inequality $\tc(X) < \cat(\double{X})$.
This answers two conjectures on topological complexity for finite spaces raised by K.~Tanaka in \cite{MR3773738}.
\end{abstract}

\maketitle

The topological complexity was first introduced by M.~Farber in \cite{MR1957228} as a numerical homotopy invariant in a relation with the robot motion planning problem. 
It attracts many authors such as Mark Grant, Jesus Gonz\'alez, Don Davis, Lucil Vandembroucq, Alexander Dranishnikov, Norio Iwase and etc.
Recently, a number of authors defined versions of topological complexity for a discrete or a finite space. For example, a combinatorial complexity is introduced by K.~Tanaka \cite{MR3773738}, a discrete topological complexity is introduced by D.~Fern\'andez-Ternero, E.~Mac\'ias-Virg\'os, E.~Minuz, and J.~A.~Vilches \cite{MR3834677} and a simplicial complexity is introduced by J.~Gonz\'alez \cite{MR3778506}. 

Topological complexity of a space $X$ is defined as follows:
for a space $X$, let $PX$ be the path space on $X$, or more precisely, all continuous maps from the unit interval $[0,1]$ to $X$, which is equipped with a projection $\pi : PX \to \double{X}$ given by $\pi(\gamma) = (\gamma(0), \gamma(1))$.
The topological complexity $\tc(X)$ of a space $X$ is the minimal number $m \geq 0$ such that $\double{X}$ is covered by $m\!+\!1$ open subsets each on which there is a section to $\pi$, i.e., section-categorical (cf. James \cite{MR1361912}).
Let $\pi_{0} : LX \to X$ be the restriction of $\pi$ to $\{\ast\}\!\times\!X \homeo X$.
Then we obtain a similar but a more classical invariant called L-S category $\cat{X}$, which is the minimal number $m \!\geq\! 0$ such that $X$ is covered by $m\!+\!1$ open subsets on which there is a section to $\pi_{0}$. Since $LX$ is contractible, each open set of the covering is contractible in $X$, in other words, categorical. 
Then the following inequality is well known:
$$
\cat(X) \le \tc(X) \le \cat(\double{X}) \le (\cat(X)\!+\!1)^{2}-1.
$$

In this paper, we consider the topological complexity and the L-S category of a finite space, though it is completely determined by the combinatorial structures. 

Let us denote by $\Cir^1_n$ the Khalimsky circle of $2n$-points, for $n \!\ge\! 2$:
$$
\Cir^1_n = \{ a_0, a_1, \ldots , a_{n-1}; b_0, b_1, \ldots , b_{n-1} \},
$$
whose topology is determined by the open sub-basis $\{a_i,  b_i, a_{i+1}\}$ where $i$ runs over $\integral/n$. Then $\mathcal{K}(\Cir^1_n)$ is a circle complex with $n$ zero-cells and $n$ one-cells. 
We remark that $\Cir^1_n$ is weakly homotopy equivalent with the topological circle $S^{1}$, more precisely, there is a weak homotopy equivalence from $S^{1}=|\mathcal{K}(\Cir^1_n)|$ to $\Cir^1_n$.
So, Khalimsky circles are good example to observe the difference between the topological complexity of the realization and that of a finite space.

Recently, Tanaka showed $\tc(\Cir^1_2)=\cat(\double{\Cir^1_2})=3$ and $\tc(\Cir^1_3)=\cat(\double{\Cir^1_3})=2$ in \cite{MR3773738}, while we know $\tc(S^{1})=1$ and $\cat(\double{S^{1}})=2$ for a topological circle $S^{1}$.
More recently, S.~Kandola showed $\cat(\Cir^1_n \times \Cir^1_n) = 2$ for all $n \geqq 3$ in \cite{MR4035463}.
In \cite{MR3773738}, K.~Tanaka raised two conjectures on the topological complexity of finite spaces.
\begin{con}\label{con:tanaka}
\begin{enumerate}
\item\cite[Conjecture 4.16]{MR3773738}\label{con:one}
$\tc(\Cir^1_{2^{k}}) = 2$  for any $k \!\ge\! 1$.
\item\cite[Conjecture 3.11]{MR3773738}\label{con:two}
There is a finite space $X$ satisfying $\tc(X) < \cat (\double{X})$.
\end{enumerate}
\end{con}
The purpose of this paper is to answer the above questions. More precisely, we answers negative to (\ref{con:one}) and positive to (\ref{con:two}) by giving the following result.

\begin{thm} \label{tc cir}
The combinatorial complexity of Khalimsky circle $\Cir^1_n$ is given as follows:
	$$ \tc(\Cir^1_n)= \begin{cases} 3, & n = 2,  \\ 2, & n = 3, 4, \\ 1, & n \geq 5. \end{cases} $$
\end{thm} 

\section{Finite $T_{0}$-spaces}\label{sect:background}
 
A finite space has a finite $T_{0}$-subspace as its strong deformation retract (see J.~A.~Barmak \cite[Remark 1.3.2]{MR3024764} for example).
So from a homotopy-theoretical view point, there are no difference between a finite space and a finite $T_{0}$-space.
From now on, we work in the category of finite $T_{0}$-spaces and continuous maps between them, unless otherwise stated. 

Firstly, P.~Alexandroff pointed out the one-to-one correspondence between a finite $T_{0}$-space and a finite poset in \cite{MR4125}: for a given finite $T_{0}$-space $X$ with a family of open sets $\mathcal{O}$ as its topology, we define a relation $\prec_{\mathcal{O}}$ as follows: for $x, y \in X$, 
$$
x \prec_{\mathcal{O}} y \overset{\text{def}}\iff 
\text{for any $U \in \mathcal{O}$, $y \in U$ implies $x \in U$.}
$$
Then $\prec_{\mathcal{O}}$ gives a poset structure on $X$, since $\mathcal{O}$ is a $T_{0}$-topology.
Conversely, for a finite poset $P$ with poset structure $\prec$, we define a topology $\mathcal{O}_{\prec}$ on $P$: for $V \subset P$,
$$
V \in \mathcal{O}_{\prec} \overset{\text{def}}\iff 
\text{$x \in V$ iff $x \prec y$ for some $y \in V$.}
$$
Then $\mathcal{O}_{\prec}$ is a $T_{0}$-topology on $P$.
So a finite $T_{0}$-space is a finite poset, and vice versa.
This implies that a homotopy-theoretical properties of finite spaces are described combinatorially.

Secondly, C.~McCord showed a tight connection between a finite poset and a finite abstract simplicial complex (asc for short) in \cite{MR1793197}: for a finite poset $P$, we obtain an order complex $\mathcal{K}(P)$ such that $\sigma \in \mathcal{K}(P)$ $\iff$ $\sigma$ is a linearly ordered subset of $P$. 
Conversely, for a finite simplicial complex $K$, we obtain a finite poset $\chi(K)=(K,\prec)$ such that $\tau \prec \sigma$ $\iff$ $\tau \subset \sigma$.
We remark that $\mathcal{K}(\chi(K))$ is nothing but the barycentric subdivision of $K$.

In \cite{MR1793197}, McCord further obtained that, for a finite $T_{0}$-space $X$ which is in fact a poset, there is a weak homotopy equivalence from $|\mathcal{K}(X)|$ to $X$, and also that, for a finite simplicial complex $K$, there is a weak homotopy equivalence $|K| \to \chi(K)$. But unfortunately, it does not imply $\tc(|\mathcal{K}(X)|) = \tc(X)$ nor $\cat(|\mathcal{K}(X)|) = \cat(X)$. 
For instance, we have $\tc(S^{1})=1 \not= 3 = \tc(\Cir^1_2)$ and $\cat(\double{S^{1}})=2 \not= 3 = \cat(\double{\Cir^1_2})$. 

Even for maps $f,g: X \to Y$ between finite $T_0$-spaces, we adopt the usual notation of homotopy, i.e. $f \simeq g$ if and only if there exists a continuous map $H : X \times I \to Y$ such that $H(x,0) = f(x)$ and $H(x,1) = g(x)$.
In that case, we say $f$ and $g$ are homotopic as usual.

\begin{thm}[Stong]
Let $f,g: X \to Y$ be two maps between finite $T_0$-spaces.
Then, $f \simeq g$ if and only if there exists a sequence of maps $f_0, f_1, \ldots, f_n$ such that
\[f = f_0 \prec f_1 \succ f_2 \prec \cdots f_n = g.
\]
\end{thm}

\begin{thm}[McCord]
Let $f,g: X \to Y$ be two homotopic maps between finite $T_0$-spaces.
Then, two maps $|\mathcal{K}(f)|, |\mathcal{K}(g)| : |\mathcal{K}(X)| \to |\mathcal{K}(Y)|$ are homotopic.
\end{thm}

Thirdly, let $X$ be a finite $T_{0}$-space which is naturally a poset.
Let us denote by
$$
[a,b]_{X}=\{x \in X \vert a \prec x \prec b\},\quad [a,-]_{X}=\{x \in X \vert a \prec x\}\ \text{and} \ [-,b]_{X}=\{x \in X \vert x \prec b\}.
$$
We abbreviate as $[a,b] = [a,b]_{X}$, $[a,-] = [a,-]_{X}$ and $[-,b] = [-,b]_{X}$, if there are no confution.

Following \cite{MR3024764}, we say that a point $x \in X$ is an up beat point if $[x,-] \smallsetminus \{x\}$ has a minimum element, and dually that a point $x \in X$ is a down beat point if $[-,x] \smallsetminus \{x\}$ has a maximum element.
An up or down beat point is said to simply be a beat point.

Note that, if $x$ is a beat point of $X$, then $X \smallsetminus \{x\}$ is a strong deformation retract of $X$.
\begin{dfn}
A finite $T_{0}$-space is said to be minimal, if it has no beat point.
A core of a finite space $X$ is a strong deformation retract of $X$, and is a minimal finite $T_{0}$-space.
\end{dfn}

Stong showed the following result on finite spaces.

\begin{thm}
A homotopy equivalence between minimal finite spaces is a homeomorphism.
In particular the core of a finite space is unique up to homeomorphism, and two finite spaces are homotopy equivalent if and only if they have homeomorphic cores.
\end{thm}

If $x$ is a beat point, we say that there is an {\it elementary strong collapse} from $X$ to $X \smallsetminus \{x\}$ and write $X \eStrColl X\smallsetminus\{x\}$ or that there is an elementary strong expansion from $X\smallsetminus\{x\}$ to $X$ and write $X\smallsetminus\{x\} \eStrExp X$.
A strong collapse $X \StrColl Y$ or a strong expansion $Y \StrExp X$ is a sequence of elementary collapses starting in $X$ and ending in $Y$.
By definition, we have the following theorem.
\begin{thm}
Let $A \subset X$. Then $X \StrColl A$ if and only if $A$ is a strong deformation retract of $X$.
\end{thm}

For an open cover $\mathcal{U}$ of a finite space $X$, we say that $\mathcal{U}$ is {\it principal} if for each open set $U$ in $\mathcal{U}$, there exists some maximal elements $m_0, m_1, \ldots, m_l$ in $X$ such that 
$$
U=\bigcup_{i} [-,m_i]_{X}.
$$

\begin{thm}
Let $X$ be a finite space of $\tc(X)=k$.
Then there exists a principal open covering of $\double{X}$ consisting of $k+1$ section-categorical open subsets.
\end{thm}

\begin{proof}
For a family of subsets $\mathcal{V}$ of a finite topological space $Y$, we define 
$$
\mathcal{V}_{M} = \{V_{M} \mid V \in \mathcal{V}\},
$$
where $V_{M} = \bigcup_{m \in M(Y) \cap V} [-,m]$.
We remark that $V_{M}$ is a subset of $V$ if $V$ is open.
Then, for any open covering $\mathcal{O}$ of $Y$, $\mathcal{O}_{M}$ is also an open covering of $Y$, and is a subdivision of $\mathcal{O}$.
In our case, we have an open covering $\mathcal{V}=\{V_0,V_1\ldots , V_k\}$ of $\double{X}$, since $\tc(X)=k$.
Thus $\mathcal{U}=\mathcal{V}_M$ gives our desired open covering of $\double{X}$ for a finite space $X$.
\end{proof}

\def\connect(#1,#2)(#3,#4){\TorusCol=#1\TorusRow=#2
\multiply\TorusCol by \number\BaseUnit
\multiply\TorusRow by \number\BaseUnit
\put(\number\TorusCol,\number\TorusRow){\line(#3,#4){\number\BaseUnit}}}
\def\multiconnect(#1,#2)(#3,#4)#5(#6,#7){
\TorusCol=#1\TorusRow=#2
\multiply\TorusCol by \number\BaseUnit
\multiply\TorusRow by \number\BaseUnit
\VectorCol=#3\VectorRow=#4
\multiply\VectorCol by \number\BaseUnit
\multiply\VectorRow by \number\BaseUnit
\multiput(\number\TorusCol,\number\TorusRow)(\VectorCol,\VectorRow){#5}{\line(#6,#7){\number\BaseUnit}}}
\def\hpoints#1{
\put(\number\TorusCol,-15){\makebox(0,0)[cc]{#1}}
\advance\TorusCol 30}
\def\vpoints#1{
\put(-15,\number\TorusRow){\makebox(0,0)[cc]{#1}}
\advance\TorusRow 30}
\def\lpoints(#1,#2)#3{\TorusCol=#1\TorusRow=#2
\multiply\TorusCol by \number\BaseUnit
\multiply\TorusRow by \number\BaseUnit
\put(\number\TorusCol,\number\TorusRow){
\makebox(0,0)[cc]{\includegraphics[width=4mm]{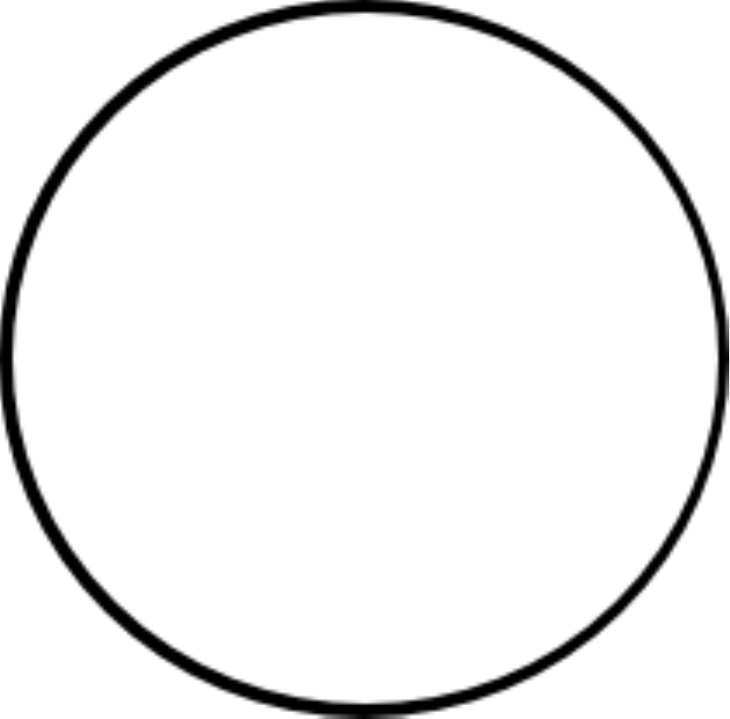}}}
\put(\number\TorusCol,\number\TorusRow){
\makebox(0,0)[cc]{\small\ifx1#3\color{green}\else\ifx2#3\color{blue}\else\ifx3#3\color{red}\else\color{black}\fi\fi\fi#3}}}

\def\multilpoints(#1,#2)(#3,#4)#5#6{
\TorusCol=#1\TorusRow=#2
\multiply\TorusCol by \number\BaseUnit
\multiply\TorusRow by \number\BaseUnit
\VectorCol=#3\VectorRow=#4
\multiply\VectorCol by \number\BaseUnit
\multiply\VectorRow by \number\BaseUnit
\multiput(\number\TorusCol,\number\TorusRow)(\VectorCol,\VectorRow){#5}{\makebox(0,0)[cc]{\includegraphics[width=4mm]{circle}}}
\multiput(\number\TorusCol,\number\TorusRow)(\VectorCol,\VectorRow){#5}{\makebox(0,0.1)[cc]{\small\ifx0#6\color{green}\else\ifx1#6\color{blue}\else\ifx2#6\color{red}\else\color{black}\fi\fi\fi#6}}}
\newenvironment{Torus}[2]
{\newcount\BaseUnit\BaseUnit=30
\newcount\IndexSizeX\IndexSizeX=#1
\newcount\IndexSizeY\IndexSizeY=#2
\newcount\LineSizeX
\newcount\LineSizeY
\newcount\TorusSizeX
\newcount\TorusSizeY
\LineSizeX=\number\BaseUnit
\LineSizeY=\number\BaseUnit
\multiply\LineSizeX by \number\IndexSizeX
\multiply\LineSizeY by \number\IndexSizeY
\TorusSizeX=\number\LineSizeX
\TorusSizeY=\number\LineSizeY
\advance\TorusSizeX by 20
\advance\TorusSizeY by 20
\advance\IndexSizeX by 1
\advance\IndexSizeY by 1
\newcount\TorusRow\newcount\TorusCol
\TorusRow=0\TorusCol=0
\newcount\VectorRow\newcount\VectorCol
\VectorRow=0\VectorCol=0
\begin{picture}(\number\TorusSizeX,\number\TorusSizeY)(-20,-20)
\curvedashes[.4mm]{1,1,1}
\multiput(0,0.1)(0,\number\BaseUnit){\number\IndexSizeY}	{\curve(0,0 , \number\LineSizeX,0)}
\multiput(0,0)(\number\BaseUnit,0){\number\IndexSizeX}	{\curve(0,0 , 0,\number\LineSizeY)}
\linethickness{1pt}
}{\end{picture}}

\begin{figure}
$$\begin{Torus}{6}{6} 
\hpoints{$a_{0}$}
\hpoints{$b_{0}$}
\hpoints{$a_{1}$}
\hpoints{$b_{1}$}
\hpoints{$a_{2}$}
\hpoints{$b_{2}$}
\hpoints{$a_{0}$}
\vpoints{$a_{0}$}
\vpoints{$b_{0}$}
\vpoints{$a_{1}$}
\vpoints{$b_{1}$}
\vpoints{$a_{2}$}
\vpoints{$b_{2}$}
\vpoints{$a_{0}$}
\newcount\inti 
\inti=1 \loop\ifnum\inti<7 
\multiconnect(1,\number\inti)(2,0){3}(1,0)
\multiconnect(1,\number\inti)(2,0){3}(1,1)
\multiconnect(1,\number\inti)(2,0){3}(0,1)
\multiconnect(1,\number\inti)(2,0){3}(-1,1)
\multiconnect(1,\number\inti)(2,0){3}(-1,0)
\multiconnect(1,\number\inti)(2,0){3}(-1,-1)
\multiconnect(1,\number\inti)(2,0){3}(0,-1)
\multiconnect(1,\number\inti)(2,0){3}(1,-1)
\advance\inti by 2\repeat
\inti=0 \loop\ifnum\inti<8 
\multiconnect(0,\number\inti)(1,0){6}(1,0)
\multiconnect(\number\inti,0)(0,1){6}(0,1)
\advance\inti by 2\repeat

\inti=1 \loop\ifnum\inti<7 
\multilpoints(0,\number\inti)(2,0){4}{1}
\multilpoints(1,\number\inti)(2,0){3}{2}
\advance\inti by 2\repeat
\inti=0 \loop\ifnum\inti<8 
\multilpoints(0,\number\inti)(2,0){4}{0}
\multilpoints(1,\number\inti)(2,0){3}{1}
\advance\inti by 2\repeat
\end{Torus}$$
\caption{The poset struture of $\double{\Cir^1_k}$}
\end{figure}

\section{The proof of $\tc(\Cir^1_k)=1, \,k \geq 5$}

The following theorem is the key to prove that $\tc(\Cir^1_k)=1, \,k \geq 5$, and shall be proved in the appendix.

\begin{thm}\label{cir map}
For given two distinct maps $f, \,g : \Cir^1_m \to \Cir^1_n$, $f$ and $g$ are homotopic if and only if $\deg f = \deg g =d $ for some $d$ with $|d| < m/n$.
\end{thm}

This immediately implies the following theorem.

\begin{thm}
 $\tc(\Cir^1_k)=1, \,k \geq 5.$
\end{thm}

\begin{proof}
Let $k \ge 5$.  It suffices to show that $\TC(\Cir^1_k) \leq 1$, since we know $1=\tc(|\Cir^1_k|) \leq \tc(\Cir^1_k)$.  From now on, let us denote the elements $a_{i}$ and $b_{i}$ in $\Cir^1_k$ by just integers $2i{+}1$ and $2i{+}2$, respectively.
Firstly, we introduce an open set $U \subset \double{\Cir^1_k}$.
$$
U = A_0 \cup A_1 \cup A_2 \cup A_3 \cup A_4,
$$
where $A_{i}$'s are defined as follows.
\begin{align*}&
A_0 = \{1,2,\ldots,2k{-}1\} {\times} \{1,2,3\}
\\&
A_1 = \{m,m{+}1,m{+}2\} {\times} \{3,4,\ldots,2k{-}1\}
\\&
A_2 = \{m{+}2,m{+}3,\ldots,2k,1\} {\times} \{2k{-}3,2k{-}2,2k{-}1\}
\\&
A_3 = \{1,2,\ldots,m{-}2\} {\times} \{5,6,\ldots, 2k{-}1\}
\\&
A_4 = \{1,2,3\} {\times} \{2k{-}1,2k,1\}
\end{align*}
where $m = \begin{cases}k, & k : \text{odd} \\ k{+}1, & k : \text{even} \end{cases}$.

Let $V$ be the smallest open neighbourhood of $\double{\Cir^1_k} \setminus U$.  We finish the proof by showing that $\pi_1 |_U \simeq \pi_2 |_U$ and $\pi_1 |_V \simeq \pi_2 |_V$.

Secondly, for each non-negative integer $i \leq 2k{-}m{-}3$, we define a continuous map $f_{i} : U \to U$ by
\[f_i(x,y)= 
  \begin{cases}
    (k_i,y), & (x,y) \in A_0 \cap \{(x,y) \mid k_i \leq x\} \subset U,
    \\ (l_i,y), & (x,y) \in A_3 \cap \{(x,y) \mid \max\{3,l_i\} \leq x\} \subset U,
    \\ (x,y), & \text{otherwise.}
  \end{cases}
\]
where $k_i = 2k{-}1{-}i$, $l_i = m{-}2{-}i$ so that $(k_i, y) \in A_0 \subset U$ if $(x,y) \in A_0$ and $(l_i, y) \in A_3 \subset U$ if $(x,y) \in A_3$.
Since $A_0$ and $A_3$ are disjoint, so $f_i$ is well-defined.

Then we can easily see that $f_i \prec f_{i{+}1}$ or $f_i \succ f_{i{+}1}$, and hence $f_i \simeq f_{i{+}1}$.
Thus we have $1_{U}=f_0 \simeq f_{2k{-}m{-}3}$ and hence we obtain that $U_1 = f_{2k{-}m{-}3}(U)$ is a deformation retract of $U$.
Since $k \ge 5$, we have $m{+}2 \le k{+}3 \le 2k{-}1$ and $3 \le k \le m$, and we can easily see the following equation.
\[C_1 = A_0' \cup A_1 \cup A_2 \cup A_3' \cup A_4 \cup A_5\]
where $A'_{0} \subset A_{0}$ and $A'_{3} \subset A_{3}$ are given as follows.
\begin{align*}&
A_0' = \{1,2,\ldots,m{+}2\}{\times}\{1,2,3\} \subset A_{0}
\\&
A_3' = \{1,2,3\}{\times}\{5,6,\ldots,2k{-}1\} \subset A_{3}.
\end{align*}
Thirdly, for each non-negative integer $0\leq i \leq 2k{-}8$, we define a continuous map $g_i:C_1 \to C_1$ by
\[g_i(x,y)=
  \begin{cases}
    (x,5{+}i), & (x,y) \in A_3' \cap \{(x,y) \mid y\leq 5{+}i\} \subset C_1,
    \\ (x,y), & \text{otherwise.}
  \end{cases}
\]
Since  $(x, 5{+}i) \in A_3' \subset C_1$ provided that $(x,y) \in A_3'$, $g_i$ is well-defined.

Then we can easily see that $g_i \simeq g_{i+1}$.  Thus we have $1_{C_1}=g_0 \simeq g_{2k-8}$, and hence we obtain that $C_2 = g_{2k{-}8}(C_1)$ is a deformation retract of $U_1$.

Since $k \ge 5$, we have $5 \le 2k{-}3$, and we can easily see the following equation.
\[C_2  = A_0' \cup A_1 \cup A_2 \cup A_3'' \cup A_4 \cup A_5 \]
where $A''_{3} \subset A'_{3}$ are given as follows.
\[A_3'' = \{1,2,3\}{\times}\{2k{-}3,2k{-}2,2k{-}1\}\]

Fourthly, we define a continuous map $h_0:C_2 \to C_2$ by 
\[h_0(x,y)= 
  \begin{cases}
    (x{+}1,y), & (x,y) \in A^{\rightarrow},
    \\ (x{-}1,y), & (x,y) \in A^{\leftarrow},
    \\ (x,y{+}1), & (x,y) \in A^{\uparrow},
    \\ (x,y{-}1), & (x,y) \in A^{\downarrow},
    \\ (x{-}1,y{+}1), & (x,y) \in A^{\nwarrow},
    \\ (x{+}1,y{-}1), & (x,y) \in A^{\searrow},
    \\ (x,y), & \text{otherwise,}
  \end{cases}
  \]
where the sets $A^{\rightarrow}, \,A^{\leftarrow}, \,A^{\uparrow}, \,A^{\downarrow}, \,A^{\nwarrow}, \,A^{\searrow}$ are defined as follows:
\begin{align*}&
A^{\rightarrow} := (\{1\}{\times}\{1,2,2k\}) \cup (\{m\}{\times}\{4,5,\ldots,2k{-}2\}),
\\&
A^{\leftarrow} := (\{3\}{\times}\{2k{-}1,2k\}) \cup (\{m{+}2\}{\times}\{2,3,\ldots,m{+}1\}),
\\&
A^{\uparrow} := (\{4,5,\ldots,m{+}1\}{\times}\{1\}) \cup (\{1,2\}{\times}\{2k{-}3\}) \cup (\{m{+}3,m{+}4,\ldots,2k\}{\times}\{2k{-}3\}),
\\&
A^{\downarrow} := (\{2,3,\ldots,m{-}1\}{\times}\{3\}) \cup (\{m{+}1,m{+}2,\ldots,2k\}{\times}\{2k{-}1\}),
\\&
A^{\nwarrow} := \{(m{+}2,1),(3,m{+}2)\},
\\&
A^{\searrow} := \{(1,3),(m,2k{-}1)\}.
\end{align*}
Then we can easily see that $1_{U_2} \simeq h_0$, and hence we obtain that $C_3 = h_0(C_2)$ is a deformation retract of $C_2$.

Finally, we define a continuous map $h_1 : C_3 \to C_3$ by 
\[h_1(x,y)= 
  \begin{cases}
    (1,2k{-}1), & (x,y) = (1, 2k{-}2),(2,2k{-}2),(2,2k{-}1),
    \\ (3,1), & (x,y) = (2,1),(2,2),(3,2),
    \\ (m,3), & (x,y) = (m,2),(m{+}1,2),(m{+}1,3),
    \\ (m{+}2,m{+}2), & (x,y) = (m{+}1,m{+}2),(m{+}1,m{+}3),(m{+}2,m{+}3),
    \\ (x,y) & \text{otherwise.}
  \end{cases}
  \]

Then we can easily see that $1_{C_3} \simeq h_1$, and hence we obtain that $C = h_1(C_2)$ is a deformation retract of $C_3$, where $C$ can be described as follows.
\begin{align*}
C = &\{(a,a{-}2) \mid a = 3,4\}\cup\{(b,2) \mid b = 4, \ldots, m{-}1\}
\\ & \cup\{(c,c{+}3{-}m) \mid c = m{-}1,m,m{+}1\} \cup\{(m{+}1,d) \mid d = 4, \ldots, m{+}1\}
\\ & \cup\{(e,e) \mid e = m{+}1,m{+}2,m{+}3\} \cup\{(f,m{+}3) \mid f = m{+}3, \ldots, 2k\}
\\ & \cup\{(g,g{+}2k{-}2) \mid g = 1,2\}.
\end{align*}
Then the arguments given above show that $C$ is a deformation retract of $U$.

Since $\deg(\pi_1|_{C}) = \deg(\pi_2|_{C}) = 1$ and $U' \homeo \Cir^1_{2k{+}m{+}2}$, we can proceed to obtain that $\pi_1|_C \simeq \pi_2|_C$ by Theorem \ref{cir map}.
A similar argument for $V$ works to deduce $\pi_1|_V \simeq \pi_2|_V$, and hence we obtain that $\tc(\Cir^1_k)=1$.

The Figure 2 below shows how the open set $U$ in the first figure is deformed into the core $C$ as a retract in the last figure, in the cace when $k=6$. 
\begin{figure}[htbp]
  \begin{center} \includegraphics[width=15cm]{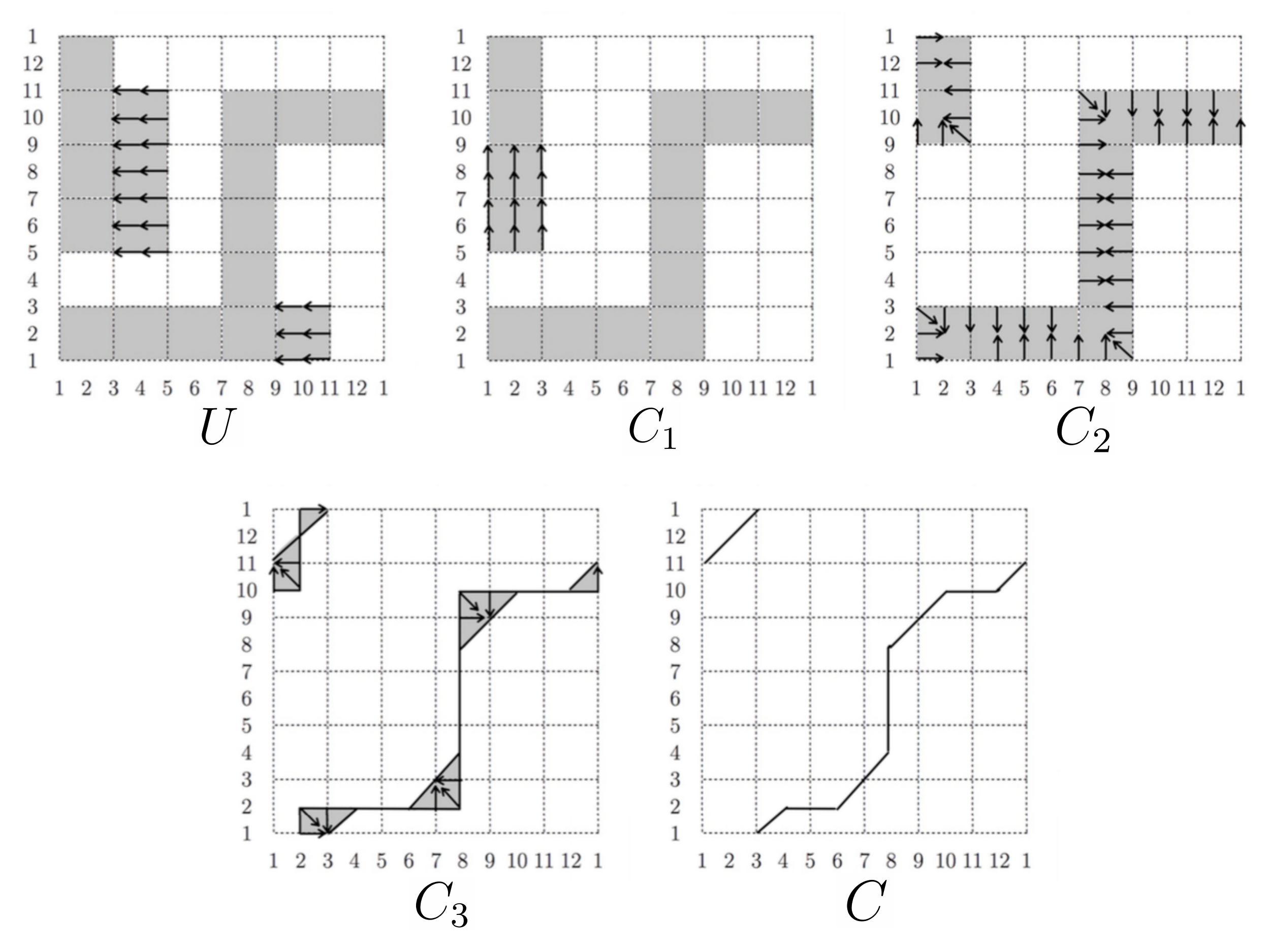}  \end{center}
  \caption{}
\end{figure}
\end{proof}

\section{The proof of $\tc(\Cir^1_4)=2$}
To prove $\tc(\Cir^1_4)=2$, we use a square-shaped figure of a Khalimsky torus $\double{\Cir^1_4}$, where $|\mathcal{K}(\double{\Cir^1_4})| \homeo \double{S^1}$, decomposed as a collection of small squares.

$$
\begin{Torus}{4}{4}
\newcount\inti
\inti=0 \loop\ifnum\inti<4
\newcount\intj \newcount\intk
\intj=\inti
\multiply \intj by 30
\intk=\intj
\advance \intk by 15
\put(-15,\number\intj){\makebox(0,0)[cc]{$a_{\number\inti}$}}
\put(-15,\number\intk){\makebox(0,0)[cc]{$b_{\number\inti}$}}
\put(\number\intj,-15){\makebox(0,0)[cc]{$a_{\number\inti}$}}
\put(\number\intk,-15){\makebox(0,0)[cc]{$b_{\number\inti}$}}
\multiconnect(0,\number\inti)(1,0){5}(0,1)
\multiconnect(\number\inti,0)(0,1){5}(1,0)
\advance\inti by 1\repeat
\put(-15,120){\makebox(0,0)[cc]{$a_{0}$}}
\put(120,-15){\makebox(0,0)[cc]{$a_{0}$}}
\put(60,-35){\makebox(0,0)[cc]{\text{\textsc{Figure 2.}\, The square decomposition}}}
\end{Torus}
$$\vskip4ex
In the square-shaped figure, each small square represents the following open set in $\double{\Cir^1_4}$:

$$\begin{Torus}{2}{2} 
\hpoints{$a$}
\hpoints{$b$}
\hpoints{$a'$}
\vpoints{$a$}
\vpoints{$b$}
\vpoints{$a'$}
\inti=1 \loop\ifnum\inti<3 
\multiconnect(1,\number\inti)(2,0){1}(1,0)
\multiconnect(1,\number\inti)(2,0){1}(1,1)
\multiconnect(1,\number\inti)(2,0){1}(0,1)
\multiconnect(1,\number\inti)(2,0){1}(-1,1)
\multiconnect(1,\number\inti)(2,0){1}(-1,0)
\multiconnect(1,\number\inti)(2,0){1}(-1,-1)
\multiconnect(1,\number\inti)(2,0){1}(0,-1)
\multiconnect(1,\number\inti)(2,0){1}(1,-1)
\advance\inti by 2\repeat
\inti=0 \loop\ifnum\inti<4 
\multiconnect(0,\number\inti)(1,0){2}(1,0)
\multiconnect(\number\inti,0)(0,1){2}(0,1)
\advance\inti by 2\repeat
\inti=1 \loop\ifnum\inti<3 
\multilpoints(0,\number\inti)(2,0){2}{1}
\multilpoints(1,\number\inti)(2,0){1}{2}
\advance\inti by 2\repeat
\inti=0 \loop\ifnum\inti<4 
\multilpoints(0,\number\inti)(2,0){2}{0}
\multilpoints(1,\number\inti)(2,0){1}{1}
\advance\inti by 2\repeat
\end{Torus}$$

Then a principal open cover of $\double{\Cir^1_4}$ by $n$-open subsets is in one-to-one correspondence with an $n$-coloring on $C$ the set of all small squares in Figure 2.
So we must consider a coloring on $C$, where an $n$-coloring on $C$ can be identified with a map $\chi : C \to \integral/n = \{0,1,\dots,n{-}1\}$. 
\begin{pro}\label{line}
If $|\chi^{-1}(i)|$ contains a whole horizontal or a whole vertical line, then $\chi^{-1}(i)$ is not sectional-categorical.
\end{pro}

\begin{proof}
In the case when $|\chi^{-1}(i)|$ contains a horizontal line, we have the subspace $V$ of $\chi^{-1}(i)$ which corresponds to the line, and $V$ is represented as the following set for some element $a \in \Cir^1_4$:
$$V = \{(x, a) \mid x \in \Cir^1_4\}.$$
Hence $V \homeo \Cir^1_4$, $\deg \pi_1|_V = 1$ and $\deg \pi_2|_V = 0$.
If we assume that $\chi^{-1}(i)$ is sectional-categorical, then $\pi_1|_{\chi^{-1}(i)} \simeq \pi_2|_{\chi^{-1}(i)}$.
Since $V \subset \chi^{-1}(i)$, we obtain $\pi_1|_V \simeq \pi_2|_V$, and hence $\deg \pi_1|_V = \deg \pi_2|_V$ by  Theorem \ref{cir map}, which is a contradiction.
In the case when $|\chi^{-1}(i)|$ contains a vertical line, a similar argument lead us to a contradiction, too.
\end{proof}

So, let us call a coloring $\chi : C \to \integral/n$ a simple coloring if $|\chi^{-1}(i)|$ does not contain a horizontal line nor a vertical line for all $i \in \integral/n$.

\begin{pro}\label{cube}
There are only two simple colorings up to isomorphism.

\ytableausetup{centertableaux}
$$
\underset{\text{\rm(I)}}{\ytableaushort{1001,0011,0110,1100}}\qquad
\underset{\text{\rm(II)}}{\ytableaushort{1011,0010,1110,1000}}
$$
\end{pro}

\begin{proof}
By the assumption on the coloring $\chi$, we have
\ytableausetup{smalltableaux}
\begin{itemize}
	\item[$(1)$] 
	In the continued two rows of squares \ytableaushort{{}{}{}{},{}{}{}{}}\:, we must find out \ytableaushort{0,0}\: and \ytableaushort{1,1}\:. \vspace{1ex}
	\item[$(2)$] 
	In the continued two columns of squares \ytableaushort{{}{},{}{},{}{},{}{}}\:, we must find out \ytableaushort{00}\: and \ytableaushort{11}\:. \vspace{1ex}
\end{itemize}

In (2) above, there are only two cases.  \ytableaushort{00}\: and \ytableaushort{11}\: are connected or not.  Forcussing on the middle two columns, we may consider the following two patterns without loss of generality.
\ytableausetup{nosmalltableaux}
$$
\underset{\text{(I)}}{\ytableaushort{{}{*(yellow)0}{*(yellow)0}{},{}{*(yellow)}{*(yellow)}{},{}{*(yellow)1}{*(yellow)1}{},{}{*(yellow)}{*(yellow)}{}}}\qquad
\underset{\text{(II)}}{\ytableaushort{{}{*(yellow)}{*(yellow)}{},{}{*(yellow)}{*(yellow)}{},{}{*(yellow)1}{*(yellow)1}{},{}{*(yellow)0}{*(yellow)0}{}}}
$$

Ignoring the upper part, either red or blue part must be \ytableausetup{smalltableaux}\ytableaushort{0,0}\: in both the cases.
\ytableausetup{nosmalltableaux}
$$
\ytableaushort{{}{*(yellow)}{*(yellow)}{} ,{}{*(yellow)}{*(yellow)}{},{*(red)}{*(yellow)1}{*(yellow)1}{*(blue)},{*(red)}{*(yellow)}{*(yellow)}{*(blue)}}
$$
Without loss of generality, we assume that the blue part is \ytableausetup{smalltableaux}\ytableaushort{0,0}\:.  Using the conditions (1) and (2), the remaining part can be filled as follows in cases (I) and (II).
\ytableausetup{nosmalltableaux}
$$
\underset{\text{(I)}}{\ytableaushort{{}{*(yellow)0}{*(yellow)0}{},{}{*(yellow)}{*(yellow)}{},{}{*(yellow)1}{*(yellow)1}{*(blue)0},{}{*(yellow)}{*(yellow)}{*(blue)0}}} \to \ytableaushort{100{},{}{}11,{}110,1{}{}0} \to \ytableaushort{100{},0011,0110,1{}{}0}  \to \ytableaushort{1001,0011,0110,11{}0}  \to \ytableaushort{1001,0011,0110,1100} 
$$
$$
\underset{\text{(II)}}{\ytableaushort{{}{*(yellow)}{*(yellow)}{},{}{*(yellow)}{*(yellow)}{},{}{*(yellow)1}{*(yellow)1}{*(blue)0},{}{*(yellow)0}{*(yellow)0}{*(blue)0}}} \to \ytableaushort{1{}{}{},{}{}{}{},1110,1000} \to \ytableaushort{1{}{}{},00{}0,1110,1000}  \to \ytableaushort{1{}11,0010,1110,1000}  \to \ytableaushort{1011,0010,1110,1000}
$$
It completes the proof of the proposition.
\end{proof}

\begin{thm}
 $\tc(\Cir^1_4)=2.$
\end{thm}

\begin{proof}
Since $1= \tc(S^1) \leq \tc(\Cir^1_4) \leq \cat(\double{\Cir^1_4})=2$, it is sufficient to show that $1<\tc(\Cir^1_4)$ where we know $1 \le \tc(\Cir^1_4) \le 2$.
By assuming $\tc(\Cir^1_4)=1$, we are lead to a contradiction.
If $\tc(\Cir^1_4)=1$, then there is an open covering $\{U_1,U_2\}$ of $\double{\Cir^1_4}$ such that $\pi_1|_{U_i} \simeq \pi_2|_{U_i}$, where $U_{i}$'s should not include a whole horizontal or a whole vertical line by Proposition \ref{line}.
Thus they must be isomorphic to (I) or (II) in Proposition \ref{cube}.
In either case, $U_{0}$ or $U_{1}$ must contain a non-trivial Khalimsky circle $\Cir \homeo \Cir^1_4$, which is not on the diagonal line.
Then by the definition of topological complexity, we have $\pi_1 |_{\Cir} \simeq \pi_2 |_{\Cir}$.
Hence $\pi_1 |_{\Cir} = \pi_2 |_{\Cir}$ by Proposition \ref{cir map}, which is contradicting that $\Cir$ is not on the diagonal line.
Thus $1<\tc(\Cir^1_4)$ and $\tc(\Cir^1_4)=2$.
\end{proof}

\appendix

\section{Homotopy classification of maps between Khalimsky Circles}

We begin this appendix with introducing a Khalimsky line denoted simply by $\integral$.  The Khalimsky line $\integral$ is the set of integers $\integral$ with the topology generated by an open basis $\{\{2k,\: 2k{+}1,\: 2k \}\mid k \in \mathbb{Z} \}$.
For two integers $k$ and $l$, we denote by $[k,l]$ the subspace  $\{z \in \integral \mid k \leqq z \leqq l\}$ of $\integral$.

Khalimsky line is known to be an Alexandroff space, and an Alexandroff space has a natural ordering and, in fact, it is in one-to-one correspondence with a poset, like a finite topological space.
Therefore, the topological properties of an Alexandroff space can be described in terms of posets.

In particular, we use a relation $\prec$ on Khalimsky line as follows : for $a,b \in \mathbb{Z}$ satisfying $|a{-}b| \leq 1$,
\[
a \prec b \overset{\text{def}} \iff a=b \text{ or } a \text{ is even.}
\]

However, we keep in mind that the relation ``$\leq$'' on $\mathbb{Z}$ is used in the usual meaning of ``greater than  or equal to''.

\begin{pro}\label{Kconti}
A function $f : [k,l] \to \mathbb{Z}$ is continuous if and only if it is an order-preserving map $f : ([k,l], \prec) \to (\mathbb{Z}, \prec)$.
\end{pro}

\begin{proof}
That is because $f|: [k,l] \to f([k,l])$ is contuinuous iff $f|$ is order-preserving.
\end{proof}

For any two continuous functions $f, g ; [k,l] \to \integral$ with $f(k)=g(k)$ and $f(l)=g(l)$, we denote by $f \sim g$, if they are homotopic relative to the both ends $k$ and $l$.
We denote by $\varepsilon_{a} : A \to \integral$ the constant function at $a$ for an Alexandroff space $A$ and $a \in \integral$.
From now on, $f : [k,l] \to \integral$ denotes a continuous function.

\begin{pro}\label{cons}
If $f : [k,l] \to \integral$ satisfies $f(k)=f(l)=a$, then $f \sim \varepsilon_{a}$.
\end{pro}

\begin{proof}
Because $f([k,l])$ is contractible in $\integral$ to $f(k)=f(l)$, we have $f\sim \varepsilon_{f(k)}$.
\end{proof}

From the above proposition we obtain the following corollary using an induction on the number of maximal points of $f$.
\begin{cor}\label{mono}
For a function $f:[k,l] \to \mathbb{Z}$, there is a monotone function $g$ such that $f \sim g$.
\end{cor}

\begin{dfn}
For a continuous function $g$ satisfying $g(k) \leq g(l)$ and $\Height(k)=\Height(g(k))$, we can define a continuous function $h_g:[k,l] \to \mathbb{Z}$ by the following formula:
$$h_g(z)=\begin{cases} z+g(k)-k, & z+g(k)-k \leq g(l), \\ g(l), & z+g(k)-k \geq g(l).\end{cases}$$
\end{dfn}

We can easily show that $h_g$ is continuous by Theorem \ref{Kconti}.
\begin{pro}\label{stan}
If $f : [k,l] \to \integral$ satisfies $f(k) \leq f(l)$ and $\Height(k)=\Height(f(k))$, then we obtain $f\sim h_f$.
\end{pro}

\begin{proof}
Without loss of generality, we may assume that $k=0$ and $f$ is monotone increasing satisfying $f(0)=0$ by Corollary \ref{mono}.
Let $f_{0}=f$.  Then by definition, $f_{0}$ is a continuous monotone increasing function, and hence we have $f_0(z{+}1) \in \{f_0(z),f_0(z){+}1\}$.
Let $k_0:=\min\{z \in [k,l] \mid f_0(z{+}1) = z\}$ and $l_0:=\max\{z \in [k,l] \mid f_0(z)=k_0\}$.
Then we have $f_0 |_{[0,k_0]} = h_f|_{[0,k_0]}$ and $k_{0} \leq f_{0}(l)$.
In case when $k_{0}=f_{0}(l)$, $f_{0}=h_{f}$ and we have done.
So we assume that $k_{0}<f_{0}(l)$. 
Then we have that $f_0(l_0{+}1)=k_0{+}1$.
Let $f_{1}, \,g_{1} : [0,l] \to \integral$ as follows:

$$f_1(z)=\begin{cases} k_0, &  z=k_0, \\ k_0+1, & z \in [k_0+1, l_0+1],\\ f_0(z), & \text{otherwise},\end{cases}$$
$$g_1(z)=\begin{cases} k_0, &  z \in [k_0,l_0+1] \text{ and } \Height(k_0)=\Height(z), \\ k_0+1, &z \in [k_0,l_0+1] \text{ and } \Height(k_0) \ne \Height(z), \\ f_0(z), & \text{otherwise}.\end{cases}$$
We then have either $f_0 \prec g_1 \succ f_1$ or $f_0 \succ g_1 \prec f_1$.
In either case, we obtain $f_0 \sim f_1$.
Let $k_1:=\min\{z \in [k,l] \mid f_1(z{+}1) = z\}$ and $l_1:=\max\{z \in [k,l] \mid f_1(z)=k_1\}$.
Then we obtain $k_1=k_0+1$.
If $k_1=f(l)$, then $f_0\sim f_1 = h_f$, and we have done.
If $k_1<f(l)$, we can continue the above process to obtain a sequence of functions $f_{1}, f_{2}, \dots, f_{n}$ for a sufficiently large $n \le l{-}k$, such that $k_{n}=f(l)$ and $f=f_{0} \sim f_{1} \sim \cdots \sim f_{n}=h_{f}$.
\end{proof}

For $m \geq 2$, the quotient space $\integral/2m = \{\overline{0},\overline{1},\dots,\overline{2m{-}1}\}$ is homeomorphic with the Khalimsky circle $\Cir^1_m$.
Moreover we have the following proposition.

\begin{lem}\label{lift}
For $m \geq 2$ and a continuous map $f: [k,l] \to \mathbb{Z}/2m$, there exists a lift $\tilde{f}$ of $f$ which makes the following diagram commutative: 
\[
\begin{tikzcd}
	\{k\} \arrow[r,"\varepsilon_a"]\arrow[d,hook] & \mathbb{Z} \arrow[d,two heads,"q"]\\
	 \Angle{k,l} \arrow[r,"f"] \arrow[ru,dashed,"{}^{\exists !}\tilde{f}"] & \mathbb{Z}/2m
\end{tikzcd}
\]
\end{lem}

\begin{proof}
We inductively define $\tilde{f}$ satisfying $q \circ \tilde{f} = f$ as follows: 
\begin{enumerate}
\item $\tilde{f}(k) = a$ and
\item for $z \in \mathbb{Z}$,
\[
\tilde{f}(z{+}1) = \begin{cases}
				\tilde{f}(z){+}1, & f(z{+}1) = f(z)+\overline{1}, \\
				\tilde{f}(z), & f(z{+}1) = f(z), \\
				\tilde{f}(z){-}1, & f(z{+}1) = f(z)-\overline{1}.
			\end{cases}
\]
\end{enumerate}
We have $q \circ \tilde{f}(k) = q(a) = q \circ \varepsilon_a(k) = f(k)$, and if $q \circ \tilde{f}(z) = f(z)$, then $q \circ \tilde{f}(z{+}1) = f(z{+}1)$ because of the definition of $\tilde{f}$. By induction, $\tilde{f}$ is a lift of $f$.

Next, we will prove the continuity of $\tilde{f}$.
We assume that $\tilde{f}$ is order-preserving on a subset $[k,a]$ of $[k,l]$.
In the case when $a \prec a{+}1$, $a$ is even.
Since $f$ is order-preserving, we have $f(a) \prec f(a{+}1)$.
So $f(a{+}1) = f(a)$ or  $f(a{+}1) = f(a) \pm \overline{1}$ and $f(a)$ is even.
If $f(a{+}1) = f(a)$, then $\tilde{f}(a{+}1)=\tilde{f}(a)$.
If $f(a{+}1) = f(a) \pm \overline{1}$, $\tilde{f}(a{+}1) = \tilde{f}(a) \pm \overline{1}$ and $\tilde{f}(a{+}1)$ is even.
Hence, $\tilde{f}(a) \prec \tilde{f}(a{+}1)$.
In the case when $a \succ a{+}1$, we similarly get $\tilde{f}(a) \succ \tilde{f}(a{+}1)$.
By induction, $\tilde{f}$ is order-preserving, i.e. $\tilde{f}$ is continuous.

Finally, we will prove the uniqueness of $\tilde{f}$.
Let $\tilde{g}$ be another lift in the commutative diagram.
We have $\tilde{f}(k)= \tilde{g}(k)$.
Suppose that $\tilde{f}|_{[k,a]}=\tilde{g}|_{[k,a]}$.
We have $\tilde{f}(a)=\tilde{g}(a)$.
$|\tilde{f}(a)-\tilde{f}(a{+}1)| \leq 1$ and $|\tilde{g}(a)-\tilde{g}(a{+}1)| \leq 1$, so $|\tilde{f}(a{+}1)-\tilde{g}(a{+}1)| \leq 2$.
Then $f(a{+}1)=g(a{+}1)$ because $q\circ f= \tilde{f} = \tilde{g}$ and $m \geq 2$.
By induction, $\tilde{f} = \tilde{g}$.
\end{proof}

From the above lemma, we obtain the following corollary by the construction of $\tilde{f}$.

\begin{cor}\label{lift_cor}
Let $m,n \geq 2$. Assume that the following diagram is commutative: 
\[
\begin{tikzcd}
	 && \mathbb{Z} \arrow[d,two heads,"q"]\\
	 \Angle{k,k+2m} \arrow[r,two heads,"p"] \arrow[rru,"\tilde{f}"] & \mathbb{Z}/2m \arrow[r,"g"] & \mathbb{Z}/2n
\end{tikzcd}
\]
Then, $\deg g = \{\tilde{f}(l) - \tilde{f}(k)\} / 2n$.
\end{cor}

\begin{pro}\label{poin}
If $f:\mathbb{Z}/2m \to \mathbb{Z}/2n$ satisfies $|\deg f|<m/n$, then it is homotopic to a continuous function $g:\mathbb{Z}/2m \to \mathbb{Z}/2n$ satisfying $g(\overline{0})=f(\overline{0})+\overline{1}$.
\end{pro}

\begin{proof}
Without loss of generality, we may assume that $f(\overline{0})=\overline{k}$ and $\deg f \geq 0$.

First, we consider the case when $\Height(0)=\Height(k)$, or equivalently, the case when $k$ is even.
There is a lift $\tilde{f}:[0,2m] \to \mathbb{Z}$ of $f \circ q$ satisfying $\tilde{f}(0)=k$ by Lemma \ref{lift}.
Then we have $\tilde{f}(k) \leq \tilde{f}(l)$ by Corollary \ref{lift_cor}, and $\tilde{f} \sim h_{\tilde{f}}$ by Proposition \ref{stan}.
For a map $h:\mathbb{Z}/2m \to \mathbb{Z}/2n$ defined by $h(\overline{z})=q \circ {h_{\tilde{f}}}(z)$, we obtain $f \simeq h$ since $\tilde{f} \sim h_{\tilde{f}}$.
From the hypothesis, we have $|\deg h| = |\deg f|< m/n$ and $h(\overline{2m{-}2})=h(\overline{2m{-}1})=h(\overline{0})=\overline{k}$.
Let $f_1,f_2:\mathbb{Z}/2m \to \mathbb{Z}/2n$ be as follows:
$$f_1(\overline{z})=\begin{cases} h(\overline{z}), & \overline{z} \ne \overline{2m{-}1}, \\ \overline{k}+\overline{1}, & \overline{z} = \overline{2m{-}1}.\end{cases}$$
$$f_2(\overline{z})=\begin{cases} f_1(\overline{z}), & \overline{z} \ne \overline{0}, \\ \overline{k}+\overline{1}, & \overline{z} = \overline{0}.\end{cases}$$
In this case, we have $f_2 \simeq h$, since $h \succ f_1 \prec f_2$.
Therefore, we obtain $f \simeq h \simeq f_2$ and hence $f_2(\overline{0})=\overline{k} + \overline{1}=f(\overline{0})+\overline{1}$.

Second, we consider the case when $\Height(0)\ne\Height(k)$, or equivalently, the case when $k$ is odd.
Then we have $\Height(2m-1) = \Height(k)$, and $f(\overline{2m-1})=f(\overline{0})=\overline{k}$ since $k$ is odd.

In this case, there is a lift $\tilde{f}:[2m-1, 4m-1] \to \mathbb{Z}$ of $f\circ q | : [2m-1, 4m-1] \to \mathbb{Z}/2n$ satisfying $\tilde{f}(2m-1)=k$, which satisfies $\tilde{f} \sim h_{\tilde{f}}$.
Then, for a map $h:\mathbb{Z}/2m \to \mathbb{Z}/2n$ defined by $h(\overline{z})=q \circ {h_{\tilde{f}}}(z)$, we obtain $f \simeq h$ since $\tilde{f} \sim h_{\tilde{f}}$.

Finally, we have $h(\overline{0})=\overline{k}+\overline{1}=f(\overline{0})+\overline{1}$, which completes the proof of this proposition.
\end{proof}

\begin{thm}
Let $f, \,g : \mathbb{Z}/2m \to \mathbb{Z}/2n$ be two distinct continuous functions.
Then $f \sim g$ if and only if $\deg f = \deg g$ and $|\deg f| <m/n$.
\end{thm}

\begin{proof}
It is sufficient to show `if' part.
So we assume that $\deg f = \deg g$ and $|\deg f| <m/n$.
By reversing the orientation of Khalimsky circle, if necessary, we may assume that $0 \leq \deg f = \deg g < m/n$.
Let $\overline{k}=f(\overline{0})$.

Without loss of generality, we may assume that $f(\overline{0}) = g(\overline{0})$ and $\Height(\overline{0}) = \Height(\overline{k})$ by Proposition \ref{poin}.
By Lemma \ref{lift}, there exist a lift $\tilde{f}$ of $f \circ p$ and a lift $\tilde{g}$ of $g \circ p$ which satisfy $\tilde{f}(0) = \tilde{g}(0) = k$.
By Corollary\ref{lift_cor}, $\tilde{f}(2m) = \tilde{g}(2m) = k + 2n \cdot \deg f$.
Thus, $\tilde{f} \sim h_{\tilde{f}} = h_{\tilde{g}} \sim \tilde{g}$ by Proposition \ref{stan}, so $f \simeq g$.

\end{proof}

\bibliographystyle{alpha}
\bibliography{Finite}

\begin{thebibliography}{FTMVMV18}

\bibitem[Ale41]{MR4125}
Paul Alexandroff.
\newblock General combinatorial topology.
\newblock {\em Trans. Amer. Math. Soc.}, 49:41--105, 1941.

\bibitem[Bar11]{MR3024764}
Jonathan~A. Barmak.
\newblock {\em Algebraic topology of finite topological spaces and
  applications}, volume 2032 of {\em Lecture Notes in Mathematics}.
\newblock Springer, Heidelberg, 2011.

\bibitem[Far03]{MR1957228}
Michael Farber.
\newblock Topological complexity of motion planning.
\newblock {\em Discrete Comput. Geom.}, 29(2):211--221, 2003.

\bibitem[FTMVMV18]{MR3834677}
D.~Fern\'{a}ndez-Ternero, E.~Mac\'{\i}as-Virg\'{o}s, E.~Minuz, and J.~A.
  Vilches.
\newblock Discrete topological complexity.
\newblock {\em Proc. Amer. Math. Soc.}, 146(10):4535--4548, 2018.

\bibitem[Gon18]{MR3778506}
Jes\'{u}s Gonz\'{a}lez.
\newblock Simplicial complexity: piecewise linear motion planning in robotics.
\newblock {\em New York J. Math.}, 24:279--292, 2018.

\bibitem[Jam95]{MR1361912}
I.~M. James.
\newblock Lusternik-{S}chnirelmann category.
\newblock In {\em Handbook of algebraic topology}, pages 1293--1310.
  North-Holland, Amsterdam, 1995.

\bibitem[Kan19]{MR4035463}
Shelley~Burrows Kandola.
\newblock {\em The {T}opological {C}omplexity of {S}paces of {D}igital
  {I}mages}.
\newblock ProQuest LLC, Ann Arbor, MI, 2019.
\newblock Thesis (Ph.D.)--University of Minnesota.

\bibitem[McC00]{MR1793197}
Christopher McCord.
\newblock Simplicial models for the global dynamics of attractors.
\newblock {\em J. Differential Equations}, 167(2):316--356, 2000.

\bibitem[Tan18]{MR3773738}
Kohei Tanaka.
\newblock A combinatorial description of topological complexity for finite
  spaces.
\newblock {\em Algebr. Geom. Topol.}, 18(2):779--796, 2018.

\end{thebibliography}
\vskip1ex

\end{document}